\documentclass[11pt,a4paper]{amsart}

\usepackage{amssymb}
\usepackage{amsthm}
\usepackage{amsmath}
\usepackage{amstext}

\numberwithin{equation}{section}

\theoremstyle{plain}
\newtheorem{thm}{Theorem}[section]

\newtheorem{cor}[thm]{Corollary}
\newtheorem{lem}[thm]{Lemma}

\theoremstyle{definition}

\newtheorem*{notation}{Notation}
\newtheorem{example}[thm]{Example}

\theoremstyle{remark}
\newtheorem{rem}[thm]{Remark}

\newcommand{\C}{\mathbb{C}}
\newcommand{\R}{\mathbb{R}}
\newcommand{\K}{\mathbb{K}}
\newcommand{\N}{\mathbb{N}}

\newcommand{\p}{\partial}
\DeclareMathOperator{\grad}{\p}

\def\({(\!(}
\def\){)\!)}

\title[ON DEFORMATION WITH CONSTANT MILNOR NUMBER]{ON  DEFORMATION WITH CONSTANT MILNOR NUMBER AND NEWTON
POLYHEDRON}

\author{Ould M Abderrahmane}

\address{D\'eparement de Math\'ematiques,Universit\'e des Sciences,
de Technologie et de Médecine BP. 880, Nouakchott, Mauritanie}

\email{ymoine@univ-nkc.mr}

\thanks{This research was supported by the Japan Society for the Promotion of Science.}

\newcommand{\abstracttext}{
 We show that every $\mu$-constant family of isolated hypersurface singularities satisfying a nondegeneracy condition in the
 sense of Kouchnirenko, is topologically
trivial, also is equimultiple.}

\begin{document}
\begin{abstract} \abstracttext \end{abstract} \maketitle

Let $f \colon (\C^n, 0) \rightarrow (\C, 0)$ be the germ of a
holomorphic function with an isolated singularity. The Milnor number
of a germ $f$, denoted by $\mu(f)$, is algebraically defined as the
$\text{dim}\,\mathcal{O}_n/{J(f)}$, where $\mathcal{O}_n$ is the
ring of complex analytic function germs $\colon (\C^n, 0) \to (\C,
0)$ and  $J(f) $ is the Jacobian ideal in $\mathcal{O}_n$ generated
by the partial derivatives $\{\frac{\grad f}{\grad
z_1},\cdots,\frac{\grad f}{\grad z_n}\}$. We recall that
 the multiplicity $m(f)$
is defined  as the lowest degree in the power series expansion of
$f$ at $0\in \C^n$. Let $F \colon (\C^n \times \C, 0) \to (\C, 0)$
be the deformation of $f$ given by $F(z, t) = f(z) + \sum
c_{\nu}(t)z^{\nu}$, where $c_{\nu}\colon (\C, 0) \to (\C, 0)$ are
germs of holomorphic functions. We use the notation $F_t(z) = F(z,
t)$ when $t$ is fixed. Let $m_t$ denote the multiplicity and $\mu_t$
denote the Milnor number of $F_t$ at the origin. The deformation $F$
is equimultiple (resp. $\mu$-constant) if $m_0 = m_t$ (resp. $\mu_0
= \mu_t$) for small $t$. It is well-known by the result of
L\^e-Ramanujam \cite{LR}. that for $n \neq 3$, the topological type
of the family $F_t$ is constant under $\mu$-constant deformations.
The question is still open for $n = 3$. However, under some
additional assumption, positive answers have been given. For
example, if $F_t$ is non-degenerate in the sense of Kouchnirenko
\cite{AGK} and the Newton boundary $\Gamma(F_t)$ of $F_t$ is
independent of $t$, i.e., $\Gamma(F_t) = \Gamma(f)$, it follows that
$\mu^{\ast}(F_t)$ is constant, and hence $F_t$ is topologically
trivial (see \cite{MO,BT} for details).
 Motivated by the Brian\c con-Speder $\mu$-constant family
$F_t(z) = z_1^5 + z_2z_3^7 + z_2^{15} + tz_1z_3^6$, which is
topologically trivial but not $\mu^{\ast}$-constant, M. Oka
\cite{MOKA} shows that any non-degenerate family of type $F(z, t) =
f(z) + tz^A$ for $A=(A_1,\dots, A_n)\in \N^{n}$, where $\N$ is the
set of nonnegative integers and
$z^{A}=z_{1}^{A_{1}}z_{2}^{A_{2}}\cdots z_{n}^{A_{n}}$ as usual, is
topologically trivial, under the assumption of $\mu$-constancy. Our
purpose of this paper is to generalize this result, more precisely,
we show that every $\mu$-constant non-degenerate family $F_t$ with
not necessarily Newton boundary $\Gamma(F_t)$ independent of $t$, is
topologically trivial. Moreover, we show that $F$ is equimultiple,
which gives a positive answer to a question of Zariski \cite{OZ,
greuel, DBO} for a non-degenerate family. To prove the main result
(Theorem \ref{manain} below), we shall use the notion of
$(c)$-regularity in stratification theory, introduced by K. Bekka in
\cite{KB}, which is weaker than  Whitney regularity, nevertheless
$(c)$-regularity implies topological triviality. First, we give a
characterization of $(c)$-regularity (Theorem \ref{critere} below).
By using it, we can show that the $\mu$-constancy condition for a
non-degenerate family implies Bekka's (c)-regularity condition and
then obtain the topological triviality as a corollary.

\begin{notation}
To simplify the notation, we will adopt the following conventions\,:
for a function $F(z, t)$ we denote by $\grad F$ the gradient of $F$
and by $\grad_z F$ the gradient of $F$ with respect to variables
$z$.

Let $\varphi,\,\, \psi \colon(\C^n, 0) \to \R$ be two function
germs. We say that $ \varphi(x)\lesssim \psi(x)$ if there exists a
positive constant $C> 0$ and an open neighborhood $U$ of the  origin
in $\C^n$ such that $\varphi(x)\leq C \; \psi(x)$, for all $x \in
U$. We write $ \varphi(x) \sim \psi(x)$  if $\varphi(x) \lesssim
\psi(x)$ and $\psi(x)\lesssim \varphi(x)$. Finally,
$|\varphi(x)|\ll|\psi(x)|$ (when $x$ tends to $x_0$) means
$\lim_{x\to x_0}\frac{\varphi(x)}{\psi(x)}=0$.

\end{notation}

\bigskip
\section{ Newton polyhedron, main results}\label{main}
\bigskip

First we recall some basic notions about the Newton polyhedron
(see\cite{AGK, MO} for details), and state the main result.

Let $f \colon (\C^n, 0) \rightarrow (\C, 0)$ be an analytic function
defined by a convergent power series $\sum c_{\nu}x^{\nu}$,  we
define $supp(f) = \{\nu\in \N^n : c_{\nu}\neq 0\}$. Also, let
$\R^n_+ = \{(x_1, \dots, x_n) \in \R^n, \text{ each } x_i \geq 0, i
= 1,\dots, n\}$. The Newton polyhedron of $f$, denoted by
$\Gamma_+(f)\subset \R^n$ is defined by the convex hull of $\{k+v
\mid k\in supp(f),\; v\in\R^n_+\}$, and let $\Gamma(f)$ be the
Newton boundary, i.e., the union of the compact faces of
$\Gamma_+(f)$. For a face $\gamma$ of $\Gamma(f)$, we write
$f_{\gamma}(z):= \sum_{\nu\in\gamma} c_{\nu}x^{\nu}$. We say that
$f$ is non-degenerate if, for any face $\gamma$ of $\Gamma(f)$, the
equations $\frac{\partial f_{\gamma}}{\partial x_1}= \dots =
\frac{\partial f_{\gamma}}{\partial x_n} = 0$ have no common
solution on $x_1=\cdots = x_n \neq 0$. The power series $f$ is said
to be convenient if $\Gamma_+(f)$ meets each of the coordinate axes.
We let $\Gamma_-(f)$ denote the compact polyhedron which is the cone
over $\Gamma(f)$ with the origin as a vertex. When $f$ is
convenient, the Newton number $\nu(f) $ is defined as $\nu(f) =
n!V_n - (n - 1)!V_{n-1} +\cdots · + (-1)^{n-1}V_1 + (-1)^n$, where
the $V_n$ are the $n$-dimensional volumes of $\Gamma_-(f)$
 and for $1 \leq k \leq n - 1$, $V_k$
is the sum of the $k$-dimensional volumes of the intersection of
$\Gamma_-(f)$ with the coordinate planes of dimension $k$. The
Newton number may also be defined for a non-convenient analytic
function (see [6]). Finally, we define the Newton vertices of $f$ as
$ver(f) = \{ \alpha : \alpha \text{ is a vertex of } \Gamma(f) \}$.

Now we can state the main result

\begin{thm}\label{manain}
 Let $F \colon
(\C^n \times \C,0) \to (\C,0)$ be a one parameter deformation of a
holomorphic germ $f \colon (\C^n, 0) \to (\C, 0)$ with an isolated
singularity such that the Milnor number $\mu(F_t)$ is constant.
Suppose that $F_t$ is non-degenerate. Then $F_t$ is topologically
trivial, and moreover, $F$ is equimultiple.
\end{thm}\medskip

\begin{rem}
In the above theorem, we do not require the independence of $t$ for
the Newton boundary $\Gamma(F_t)$
\end{rem}
\bigskip
\section{criterion for $(c)$-regularity}
\bigskip
By way of notation, we let $G(k, n)$ denote the set of
$k$-dimensional linear subspace of  the vector space $\K^n$, where
$\K=\R$ or $\C$.

  Let $M$ be a smooth manifold, and let $X$, $Y$ be smooth
submanifolds of $M$ such that $Y \subset \overline{X}$ and $X \cap Y
= \emptyset$.
\medskip
\begin{enumerate}
\item[(i)](Whitney $(a)$-regularity)\\
$(X, Y )$ is $(a)$-regular at $y_0\in  Y$ if :\\
for each sequence of points $\{x_i\}$ which tends to $y_0$ such that
the sequence of tangent spaces $\{T_{x_i}X\}$ tends  to $\tau$ in
the grassmannian $G(\text{dim} \,X,\text{dim}\, M)$, one hase
$T_{y_0}Y \subset\tau$. We say $(X, Y )$ is $(a)$-regular if it is
$(a)$-regular at any point $y_0\in Y$.

\item[(ii)] (Bekka $(c)$-regularity)\\
Let $\rho$ be a smooth non-negative function such that $\rho^{-1}(0)
= Y$. $(X, Y )$ is $(c)$- regular at $y_0 \in Y$ for the control
function $\rho$ if :\\
for each sequence of points $\{x_i\}$ which tends to $y_0$ such that
the sequence of tangent spaces $\{\text{Ker}d\rho(x_i) \cap
T_{x_i}X\}$ tends to $\tau$ in the grassmannian $G(\text{dim} \,X\,-
1,\text{dim} \,M)$, one hase $T_{y_0}Y \subset\tau$. $(X, Y )$ is
$(c)$-regular at $y_0$ if it is $(c)$-regular for some control
function $\rho$. We say $(X, Y )$ is $(c)$-regular if it is
$(c)$-regular at any point $y_0\in Y$.
\end{enumerate}
\medskip

Let $F \colon (\C^n \times\C, \{0\} \times \C) \to (\C,0)$ be a
deformation of an analytic function $f$. We denote by $\Sigma(V_F) =
\{F^{-1}(0) - \{0\}\times \C,\;\; \{0\} \times\C\}$ the canonical
stratification of the germ variety $V_F$ of the zero locus of $F$.
We may assume that $f$ is convenient, this is not a restriction when
it defines an isolated singularity, in fact, by adding $z_i^N$ for a
sufficiently large $N$ for which the isomorphism class of $F_t$ does
not change. Hereafter, we will assume that $f$ is convenient,
$$
 X=F^{-1}(0) - \{0\}\times \C, \,\,\,\, Y= \{0\} \times\C \text{ and }
 \rho(z)=\sum_{\alpha\in
 \text{ver}(F_t)}z^{\alpha}\overline{z}^{\alpha}.
$$

Here ver$(F_t)$ denotes the Newton vertices of $F_t$ when $t \neq
0$.

Note that by the convenience assumption on $f$, $\rho^{-1}(0) = Y$.

We also let
$$
\partial{\rho} =\sum_{ i=1}^n \frac{\partial{\rho}}{\partial z_i} \frac{\partial}{\partial z_i}+
\frac{\partial{\rho}}{\partial \overline{z}_i}
\frac{\partial}{\partial \overline{z}_i}=\partial_z{\rho}+
\partial_{\overline{z}}\rho
$$
and
$$
\partial{F} =\sum_{ i=1}^n \frac{\partial{F}}{\partial z_i} \frac{\partial}{\partial
z_i}+ \frac{\partial{F}}{\partial t}=\partial_z{F}+
\partial_{t}F.
$$

{\bf Calculation of the map $\partial_z\rho_{|X}$}

First of all we remark that $\grad_z\rho = \grad_z\rho_{|X} +
\grad_z\rho_{| N} $ (where $N$ denotes the normal space to $X$).
Since $N$ is generated by the gradient of $F$, we have that
$\grad_z\rho = \grad_z\rho_{|X} + \eta\grad F$. On the other hand,
$\langle \grad_z\rho_{|X}, \grad F \rangle=0$, so we get
$\eta=\frac{\langle \grad_z\rho, \grad F \rangle}{| \grad F|^2}$. It
follows that
\begin{equation}\label{control}
\grad_z\rho_{|X} =\grad_z\rho -\frac{\langle \grad_z\rho, \grad F
\rangle}{| \grad F|^2}\grad F=(\grad_z\rho_{|X})_z +
(\grad_z\rho_{|X})_{t},
\end{equation}
weher
$$
(\grad_z\rho_{|X})_z =\grad_z\rho -\frac{\langle \grad_z\rho, \grad
F \rangle}{| \grad F|^2}\grad_z F, \;\; (\grad_z\rho_{|X})_t
=-\frac{\langle \grad_z\rho, \grad F \rangle}{| \grad F|^2}\grad_t F
$$
and
$$
|\grad_z\rho_{|X}|^2=\frac{| \grad F|^2| \grad_z\rho|^2-| \langle
\grad_z\rho, \grad F \rangle|^2}{| \grad F|^2}=\frac{\|\grad F
\wedge\grad_z\rho\|^2}{| \grad F|^2}.
$$
Then we can characterise the $(c)$-regularity as follows :
\begin{thm}\label{critere}
Consider $X$ and $Y$ as above. The following conditions are
equivalent

(i) $(X, Y )$ is $(c)$-regular for the the control function $\rho$.

(ii) $(X, Y )$ is $(a)$-regular and $|(\grad_z\rho_{|X})_{t}| \ll
|\grad_z\rho_{|X}|$ as $(z, t) \in X$ and $(z, t) \to Y$.

(iii) $| \grad_t F|\ll \frac{\|\grad F \wedge\grad_z\rho\|}{|
\grad_z \rho|}$ as $(z, t) \in X$ and $(z, t) \to Y$.

\end{thm}

\begin{proof}
Since (i) $\Leftrightarrow$ (ii) is proved in (\cite{OMA}, Theorem
1), and (iii) $ \Rightarrow $ (ii) is trivial, it is enough to see
(ii) $\Rightarrow$ (iii).

To show that (ii) $ \Rightarrow$ (iii), it suffices to show this on
any analytic curve $\lambda(s) = (z(s), t(s)) \in X$ and $\lambda(s)
\to Y$. Indeed, we have to distinguish two cases :

First case, we suppose that along $\lambda$, $|\langle \grad_z\rho,
\grad F \rangle|   \sim |\grad_z\rho| |\grad F |$, hence by
(\ref{control}) and (ii), we have
$$
|(\grad_z\rho_{|X})_t| = |\frac{\langle \grad_z\rho, \grad F
\rangle}{|\grad F|^2}\grad_tF | \ll  \frac{\|\grad F
\wedge\grad_z\rho\|}{| \grad F|}.
$$
But this clearly implies
$$
| \grad_t F|\ll \frac{\|\grad F \wedge\grad_z\rho\|}{| \grad_z
\rho|} \;\; \text{ along the curve }  \lambda(s),
$$
 where
 $|\langle \grad_z\rho, \grad F \rangle|   \sim |\grad_z\rho| |\grad F|. $

Second case, we suppose that along $\lambda$, $|\langle \grad_z\rho,
\grad F \rangle|   \ll |\grad_z\rho| |\grad F |$, thus

$$
\|\grad F \wedge\grad_z\rho\| \sim |\grad_z\rho| |\grad F| \text{
along the curve } \lambda(s).
$$
 On the other hand, by
the Whitney $(a)$-regularity in (ii) we get
$$
|\grad_tF| \ll |\grad F|.
$$
Therefore, $ |\grad_tF| \ll |\grad F| \sim \frac{\|\grad F
\wedge\grad_z\rho\|}{| \grad_z \rho|}$ along the curve $\lambda(s)$.
The Theorem \ref{critere} is proved
\end{proof}

 \section{ Proof of the theorem \ref{manain} }

 Before starting the proofs, we will recall
some important results on the Newton number and the geometric
characterization of $\mu$-constancy.

\begin{thm}[A. G. Kouchnirenko \cite{AGK}]\label{kouchnirenko}
 Let $f \colon (\C^n, 0) \to (\C,0)$ be the germ
of a holomorphic function with an isolated singularity, then the
Milnor number $\mu(f) \geq \nu(f$). Moreover, the equality holds if
$f$ is non-degenerate.
\end{thm}
As an immediate corollary we have
\begin{cor}[M. Furuya \cite{MF}]\label{furuya}
 Let $f, \,\,g\colon (\C^n, 0) \to(\C,
0)$ be two germs of holomorphic functions with $\Gamma_+(g)
\subset\Gamma_+(f)$. Then $\nu(g) \geq\nu (f)$.
\end{cor}

On the other hand, concerning the $\mu$-constancy, we have

\begin{thm}[Greuel \cite{greuel}, L\^e-Saito \cite{LS}, Teissier \cite{BT}]\label{le-saito-tessier}
Let $F \colon (\C^n\times \C^m, 0) \to (\C, 0)$ be the deformation
of a holomorphic $f \colon (\C^n, 0) \to (\C, 0)$ with isolated
singularity. The following statements are equivalent.
\begin{enumerate}
\item $F$ is a $\mu$-constant deformation of $f$.

\item $\frac{\grad F}{\grad t_j}\in \overline{J(F_t)}$, where $\overline{J(F_t)}$
denotes the integral closure of
the Jacobian ideal of $F_t$ generated by the partial derivatives of
$F$ with respect to the variables $z_1,\dots, z_n$.

\item  The deformation $ F(z, t) = F_t(z)$ is a Thom map, that is,
$$
\sum_{ j=1}^m |\frac{\grad F}{\grad t_j} | \ll \| \grad F\| \text{
as } (z, t) \to (0, 0).
$$

\item  The polar curve of $F$ with respect to $\{t = 0\}$ does not split,
that is,
$$
\{(z,t) \in \C^n \times \C^m \,\,| \,\,\grad_zF (z,t) = 0\} = \{0\}
\times \C^m \text{ near } (0,0).
$$

\end{enumerate}
\end{thm}

We now want to prove  theorem \ref{manain}, in fact, let $F :
(\C^n\times\C,0) \to(\C, 0)$ be a deformation of a holomorphic germ
$f \colon (\C^n, 0)\to (\C, 0)$ with an isolated singularity such
that the Milnor number $\mu(F_t)$ is constant. Suppose that $F_t$ is
non-degenerate. Then, by theorem \ref{kouchnirenko}, we have

\begin{equation}\label{3.1}
\mu(f) = \nu(f) = \mu(F_t) = \nu(F_t).
\end{equation}

Consider the deformation $\widetilde{F}$ of f given by
$$
\widetilde{F}(z, t, \lambda) = F_t(z) +
\sum_{\alpha\in\text{ver}(F_t)} \lambda_{\alpha}z^{\alpha}.
$$
From the upper semi-continuity of Milnor number \cite{JM}, we obtain
\begin{equation}\label{3.2}
 \mu(f) \geq\mu(\widetilde{F}_{t,\lambda}) \text{ for }(t, \lambda) \text{ near } (0, 0).
\end{equation}

By Theorem \ref{kouchnirenko} and Corollary \ref{furuya} therefore
$$
\mu(\widetilde{F}_{t,\lambda}) \geq\nu(\widetilde{F}_{t,\lambda})
\geq \nu(F_t).
$$
It follows from (\ref{3.1}) and (\ref{3.2}) that the deformation
$\widetilde{F}$ is $\mu$-constant, and hence, by Theorem
\ref{le-saito-tessier} we get
\begin{equation}\label{3.3}
|\grad_t F| + \sum_{\alpha \in\text{ver}(F_t)} |z^{\alpha}| \ll
|\grad_z F + \sum_{\alpha \in\text{ver}(F_t)}
\lambda_{\alpha}z^{\alpha}| \text{ as } (z, t, \lambda) \to (0, 0,
0).
\end{equation}

Therefore, for all $\alpha\in \text{ver}(F_t)$ we have $|z^{\alpha}|
\ll |\grad_z f|$, and so $ m(z^{\alpha}) \geq m(f)$. Hence the
equality $m(F_t) = m(f)$ follows. In other words, $F$ is
equimultiple.

 We also
show that condition (\ref{3.3}), in fact, implies Bekka's
$(c)$-regularity, hence, this deformation is topologically trivial.
For this purpose, we need the following lemma (see \cite{AP}).
\begin{lem}\label{lem}
 Suppose $F_t$ is a deformation as above, then we have
\begin{equation}\label{3.4}
\sum_{\alpha \in\text{ver}(F_t)} |z^{\alpha}| \ll \inf_{\eta\in
\C}\{|\grad F + \sum_{\alpha \in\text{ver}(F_t)} \eta
\overline{z}^{\alpha}\grad_z z^{\alpha}|\} \text{ as } (z, t) \to
(0, 0), F(z, t) = 0.
\end{equation}
\end{lem}

\begin{proof}
Suppose (\ref{3.4}) does not hold. Then by the curve selection
lemma, there exists an analytic curve $p(s) = (z(s), t(s))$ and an
analytic function $\eta(s)$, $s \in [0, \epsilon)$, such that :
\item {(a)} $p(0)=0$;
\item {(b)} $F(p(s)) \equiv 0$, and hence $dF(p(s)) \frac{d p}{d s} \equiv 0$;
\item (c) along the curve p(s) we have
$$
\sum_{\alpha \in\text{ver}(F_t)} |z^{\alpha}| \gtrsim |\grad F +
\sum_{\alpha \in\text{ver}(F_t)} \eta(s)
\overline{z}^{\alpha}\grad_z z^{\alpha}|.
$$
Set
\begin{equation}\label{3.5}
g(z,\overline{z})=\left(\sum_{\alpha
\in\text{ver}(F_t)}\overline{z}^{\alpha}z^{\alpha}
\right)^{\frac{1}{2}} \;\text{ and }\; \gamma(s)=\eta(s)g(z(s),
\overline{z}(s))
\end{equation}

First suppose that $\gamma(s) \to 0$. Since $|\overline{z}^{\alpha}|
\leq g$, we have,
$$
\lambda_{\alpha}=\frac{\gamma(s)\overline{z}^{\alpha}(s)}{g(z(s),\overline{z}(s))}\to
0,\;\; \forall\alpha\in\text{ ver}(F_t)
$$
Next, using (\ref{3.3}) and (\ref{3.5}) it follows
$$
\sum_{\alpha \in\text{ver}(F_t)}| z^{\alpha}(s)|\ll |\grad F(p(s)) +
\sum_{\alpha \in\text{ver}(F_t)} \eta(s)
\overline{z}^{\alpha}(s)\grad_z z^{\alpha}(s)|\;\; \text{ as } s\to
0,
$$
which contradicts (c).

Suppose now that the limit of $\gamma(s)$ is not zero (i.e.,
$|\gamma(s)| \gtrsim 1$ ). Since $p(0) = 0$ and
$g(z(0),\overline{z}(0)) = 0$, we have, asymptotically as  $s\to 0$,
\begin{equation}\label{3.6}
 s| \frac{d p}{d s}(s)| \sim |p(s)| \text{ and } s \frac{d}{ d s}
g(z(s), \overline{z}(s)) \sim g(z(s), \overline{z}(s)).
\end{equation}
But
\begin{equation}\label{3.7}
\frac{d}{ d s} g(z(s), \overline{z}(s))
 =\sum_{\alpha \in\text{ver}(F_t)}\frac{1}{2g(z(s),
 \overline{z}(s))}\left(\overline{z}^{\alpha}d z^{\alpha}\frac{d z}{d s} + z^{\alpha}d\overline{z}^{\alpha}\frac{d\overline{z}}{d
 s}\right).
\end{equation}

We have $\overline{z}^{\alpha}d z^{\alpha}\frac{d z}{d s}
=\overline{z^{\alpha}d\overline{z}^{\alpha}\frac{d\overline{z}}{d
s}} $ and $1 \lesssim |\gamma(s)|$. Thus,
\begin{equation}\label{3.8}
| \frac{d}{ d s} g(z(s), \overline{z}(s))|\lesssim\left|\sum_{\alpha
\in\text{ver}(F_t)}\frac{\gamma(s)}{g(z(s),
 \overline{z}(s))}\overline{z}^{\alpha}d z^{\alpha}\frac{d z}{d s}
 \right|.
\end{equation}
This together with (\ref{3.6}), (\ref{3.5}) and (b) gives
$$
g(z(s), \overline{z}(s))\sim |s \frac{d}{ d s} g(z(s),
\overline{z}(s)) |\lesssim s\left| \sum_{\alpha
\in\text{ver}(F_t)}\eta(s) \overline{z}^{\alpha}d z^{\alpha}\frac{d
z}{d s} + d F(p(s))\frac{d p}{d s}\right|.
$$
Hence
$$
g(z(s), \overline{z}(s))\lesssim s\left|\frac{d p}{d s}(s)\right|\,
\left| \sum_{\alpha \in\text{ver}(F_t)}\eta(s)
\overline{z}^{\alpha}\grad z^{\alpha}+ \grad F(p(s))\right|,
$$
which contradicts (c). This ends the proof of Lemma.
\end{proof}

We shall complete the proof of Theorem \ref{manain} Since
$\Gamma_+(\grad_tF) \subset \Gamma_+(F_t)$. Then, by an argument,
based again on the curve selection lemma, we get the following
inequality
\begin{equation}\label{3.9}
|\grad_tF| \lesssim \sum_{\alpha \in\text{ver}(F_t)}|z^{\alpha}|.
\end{equation}
Then, by the above Lemma \ref{lem}, we obtain
$$
|\grad_t F|\ll\inf_{\eta\in \C}\{|\grad F +\eta\grad_z\rho|\} \text{
as } (z,t)\to (0,0),\; F(z,t)=0,
$$
we recall that
$$
\rho(z) = \sum_{\alpha
\in\text{ver}(F_t)}z^{\alpha}\overline{z}^{\alpha}.
$$
But
$$
\inf_{\eta\in \C}\{|\grad F +\eta\grad_z\rho|\}^2=\frac{| \grad
F|^2| \grad_z\rho|^2-| \langle \grad_z\rho, \grad F \rangle|^2}{|
\grad_z \rho |^2}=\frac{\|\grad F \wedge\grad_z\rho\|^2}{| \grad_z
\rho|^2}.
$$
Therefore, by Theorem \ref{critere}, we see that the canonical
stratification $\Sigma(V_F )$ is $(c)$-regular for the control
function $\rho$, then $F$ is a topologically trivial deformation
(see\cite{KB}).

This completes the proof of Theorem \ref{manain}.

\begin{rem}
We should mention that our arguments still hold for any
$\mu$-constant deformation $F$ of  a weighted homogeneous polynomial
$f$ with isolated singularity. Indeed, we can find from Varchenko's
theorem \cite{ANV} that $\mu(f) = \nu(f) = \mu(F_t) = \nu(F_t)$.
Thus, the above proof can be applied.

Unfortunately this approach does not work, if we only suppose that
$f$ is non-degenerate. For consider the example of Altman \cite{KA}
defined by
$$
F_t(x, y, z) = x^5 + y^6 + z^5 + y^3z^2 + 2tx^2y^2z + t^2x^4y,
$$
which is a $\mu$-constant degenerate deformation of the
non-degenerate polynomial $f(x, y, z) = x^5 +y^6 +z^5 +y^3z^2.$ He
showed that this family has a weak simultaneous resolution. Thus, by
Laufer's theorem \cite{HBL}, $F$ is a topologically trivial
deformation. But we cannot apply the above proof because $\mu(f) =
\nu(f) = \mu(F_t) = 68$ and $\nu(F_t) = 67$ for $t \neq 0$.
\end{rem}

We conclude with several examples.

 \begin{example}  Consider the family
given by
$$
F_t(x, y, z) = x^{13} + y^{20 }+ zx^6y^5 + tx^6y^8 + t^2x^{10}y^3 +
z^l, \;\; l \geq 7.
$$

It is not hard to see that this family is non-degenerate. Moreover,
by using the formula for the computation of Newton number we get
$\muµ(F_t) = \nu(F_t) = 153 l + 32$. Thus, by theorem \ref{manain},
we have that $F_t$ is topologically trivial. We remark that this
deformation is not $\mu^{\ast}$-constant, in fact, the Milnor
numbers of the generic hyperplane sections $\{z = 0\} $ of $F_0$ and
$F_t$  for $t\neq 0$ are $260$ and $189$ respectively.
\end{example}

\begin{example}
Let
$$
F_t(x, y, z) = x^{10} + x^3y^4z + y^l + z^l + t^3x^4y^5 + t^5x^4y^5
$$
where $l \geq 6$. Since $\mu(F_t) = 2l^2 +32l +9$ and $F_t$ is a
non-degenerate family, it follows from Theorem \ref{manain} that $F$
is a topologically trivial deformation
\end{example}




\bigskip
\begin{thebibliography}{99}

\bibitem{OMA} Ould. M. Abderrahmane, \emph{ Stratification theory from the Newton
polyhedron point of view,} Ann. Inst. Fourier, Grenoble \textbf{54},
(2004), 235--252.

\bibitem{KA}K. Altmann, \emph{Equisingular deformation of isolated
2-dimensional hypersurface singularities,} Invent. math. \textbf{88}
(1987), 619--634.

\bibitem{KB} K. Bekka, \emph{$(c)$-r\'egularit\'e et trivialit\'e topologique,
Singularity theory and its applications,} Warwick 1989, Part I,
Lecture Notes in Math. \textbf{1462} (Springer, Berlin 1991),
42--62.

\bibitem{BS} J. Brianc¸on and J.P. Speder, \emph{La trivialit\'e
topologique n'implique pas les conditions de Whitney,} C. R. Acad.
Sci. Paris \textbf{280} (1976), 365–-367.

\bibitem{greuel} G-M. Greuel, \emph{Constant Milnor Number Implies Constant
Multiplicity For Quasihomogeneous Singularities,} Manuscritpta Math.
\textbf{56} (1986), 159--166.

\bibitem{MF} M. Furuya, \emph{Lower Bound of Newton Number,} Tokyo J. Math. \textbf{27} (2004),
177-186.

\bibitem{AGK} A.G. Kouchnirenko, \emph{Poly\`edres de Newton et nombres de
Milnor,} Invent. math. \textbf{32} (1976), 1-31.

\bibitem{HBL} H. B. Laufer, \emph{Weak simultaneous resolution for deformation of
Gorenstein surface singularities}, Proc. Symp. Pure Math.
\textbf{40} part 2 (1983), 1–28.

\bibitem{LR} D. T. Leˆ and C. P. Ramanujam, \emph{Invariance of Milnor's number
implies the invariance of topological type,} Amer. J. Math.
\textbf{98} (1976), 67-78.

\bibitem{LS} D.T. Leˆ and K. Saito,
\emph{La constence du nombre de Milnor donne des bonnes
stratifications,} Compt. Rendus Acad. Sci. Paris, s\'erie A
\textbf{272} (1973), 793--795.

\bibitem{JM} J. Milnor, \emph{Singular points of complex hypersurfaces}, Ann. of Math.
Stud. \textbf{61} (1968), Princeton Univer- sity Press.

\bibitem{MO} M. Oka, \emph{On the bifurcation of the multiplicity and topology of the
Newton boundary}, J. Math. Soc. Japan \textbf{31} (1979), 435-450.

\bibitem{MOKA} M. Oka, \emph{On the weak simultaneous resolution of a negligible
truncation of the Newton boundary,} Contemporary. Math., \textbf{90}
(1989), 199–-210.

\bibitem{DBO} D. B. O'Shea, \emph{Topologically Trivial Deformations
of Isolated Quasihomogeneous Singularities Are Equimultiple,} Proc.
A.M.S. \textbf{101}:2 (1987), 260--262.
\bibitem{AP} A. Parusi\'nski, \emph{Topological triviality of $\mu$-constant deformations of
type $f(x)+tg(x)$,} Bull. London Math. Soc \textbf{31} (1999),
686-692.

\bibitem{BT} B. Teissier, \emph{Cycles \'evanescents, section planes, et conditions de
Whitney,} Singularite´s a` Carg\`ese 1972, Ast\'erisque No
\textbf{7-8}, Soc. Math. Fr. 285-362 (1973)

\bibitem{trotman} D. Trotman, \emph{Equisingularit\'e et conditions de
Whitney,}  th\`ese   Orsay, January 15, 1980

\bibitem{ANV} A. N. Varchenko, \emph{A lower bound for the codimension of the stratum
µ-constant in term of the mixed Hodge structure,} Vest. Mosk. Univ.
Mat. \textbf{37} (1982), 29--31

\bibitem{OZ}O. Zariski, \emph{Open questions in the theory of singularities,} Bull. Amer. Math. Soc \textbf{77} (1971),481-491.

\end{thebibliography}
\end{document}